\documentclass[a4paper, reqno]{amsart}
\usepackage[utf8]{inputenc}
\usepackage[T1]{fontenc}
\usepackage{lmodern}
\usepackage{amssymb}
\usepackage[all]{xy}
\usepackage{nicefrac,mathtools,enumitem}
\usepackage{microtype}
\usepackage{ifthen}
\usepackage{leftidx}
\usepackage{amssymb}
\usepackage{amsxtra}

\usepackage{tikz}
\usetikzlibrary{matrix}
\tikzset{cd/.style=matrix of math nodes,row sep=2em,column sep=2em, text height=1.5ex, text depth=0.5ex}
\tikzset{cdar/.style=->,auto}

\setlist[enumerate,1]{label=\textup{(\arabic*)}}
\setlist[enumerate,2]{label=\textup{(\alph*)}}

\usepackage[pdftitle={Faithful actions of LCQG on classical spaces},
pdfauthor={Debashish Goswami and Sutanu Roy},
pdfsubject={Mathematics}
]{hyperref}
\usepackage[lite]{amsrefs}
\newcommand*{\MRref}[2]{ \href{http://www.ams.org/mathscinet-getitem?mr=#1}{MR #1}}
\newcommand*{\arxiv}[1]{ \href{http://www.arxiv.org/abs/#1}{arXiv:#1}}

\renewcommand{\PrintDOI}[1]{\href{http://dx.doi.org/\detokenize{#1}}{doi: \detokenize{#1}}%
  \IfEmptyBibField{pages}{, (to appear in print)}{}}

\BibSpec{book}{%
    +{}  {\PrintPrimary}                {transition}
    +{,} { \textit}                     {title}
    +{.} { }                            {part}
    +{:} { \textit}                     {subtitle}
    +{,} { \PrintEdition}               {edition}
    +{}  { \PrintEditorsB}              {editor}
    +{,} { \PrintTranslatorsC}          {translator}
    +{,} { \PrintContributions}         {contribution}
    +{,} { }                            {series}
    +{,} { \voltext}                    {volume}
    +{,} { }                            {publisher}
    +{,} { }                            {organization}
    +{,} { }                            {address}
    +{,} { \PrintDateB}                 {date}
    +{,} { }                            {status}
    +{}  { \parenthesize}               {language}
    +{}  { \PrintTranslation}           {translation}
    +{;} { \PrintReprint}               {reprint}
    +{.} { }                            {note}
    +{.} {}                             {transition}
    +{,} { \PrintDOI}                   {doi}
    +{}  {\SentenceSpace \PrintReviews} {review}
}

\numberwithin{equation}{section}

\theoremstyle{plain}
\newtheorem{theorem}[equation]{Theorem}
\newtheorem{lemma}[equation]{Lemma}
\newtheorem{proposition}[equation]{Proposition}

\theoremstyle{definition}
\newtheorem{definition}[equation]{Definition}

\theoremstyle{remark}

\newtheorem{example}[equation]{Example}

\newcommand*{\dom}{\textup{Dom}}

\newcommand*{\nb}{\nobreakdash}
\newcommand*{\Star}{\(^*\)\nb-}

\newcommand*{\C}{\mathbb C}
\newcommand*{\R}{\mathbb R}
\newcommand*{\N}{\mathbb N}

\newcommand*{\diff}{\textup d}

 \newcommand*{\G}[1][G]{\mathbb #1}
\newcommand*{\DuG}[1][G]{\widehat{\mathbb{#1}}}

\newcommand*{\Comult}[1][]{\Delta_{#1}}
\newcommand*{\DuComult}[1][]{\hat{\Delta}_{#1}}
\newcommand*{\Qgrp}[2]{\mathbb #1=(#2,\Comult[#2])}
\newcommand*{\DuQgrp}[2]{\widehat{\mathbb #1}=(\hat{#2},\DuComult[#2])}

\newcommand*{\Bialg}[1]{(#1,\Comult[#1])}

\newcommand*{\Multunit}[1][]{\mathbb{W}^{#1}}
\newcommand*{\DuMultunit}[1][]{\widehat{\mathbb{W}}{}^{#1}}

\newcommand*{\CLS}{\mathrm{CLS}}

\newcommand*{\Bound}{\mathbb B}
\newcommand*{\Comp}{\mathbb K}

\newcommand*{\Contvin}{\textup C_0}
\newcommand*{\Contb}{\textup C_\textup b}
\newcommand*{\Cont}{\textup C}
\newcommand*{\CompSupp}{\textup{C}_{\textup{c}}}
\newcommand*{\Mor}{\textup{Mor}}
\newcommand*{\Id}{\textup{id}}

\newcommand*{\Flip}{\Sigma}
\newcommand*{\flip}{\sigma}
\newcommand*{\Cst}{\mathrm C^*}
\newcommand*{\Cred}{\mathrm C^*_\mathrm r}

\newcommand*{\Cstcat}{\mathfrak{C^*alg}}

\newcommand*{\Hils}[1][H]{\mathcal #1}
\newcommand*{\Mult}{\mathcal M}
\newcommand*{\U}{\mathcal U}

\newcommand*{\defeq}{\mathrel{\vcentcolon=}}
\newcommand*{\abs}[1]{\lvert#1\rvert}
\newcommand*{\norm}[1]{\lVert#1\rVert}
\newcommand*{\conj}[1]{\overline{#1}}

\begin{document}
\title[Faithful actions of LCQG on classical spaces]{Faithful actions of locally compact quantum groups on classical spaces}

\author{Debashish Goswami}
\email{goswamid@isical.ac.in}
\address{Statistics and Mathematics Unit\\
 Indian Statistical Institute\\
 203, B. T. Road, Kolkata 700108\\
 India}

\author{Sutanu Roy}
\email{sutanu.roy@carleton.ca}
\address{School of Mathematics and Statistics\\
 Carleton University \\
 1125 Colonel By Drive, Ottawa, ON K1S 5B6\\
 Canada}

\begin{abstract}
It is well-known that no non\nb-Kac compact quantum group can faithfully  act on \(\Cont(X)\) for a classical, compact Hausdorff space~\(X\). However, in this article we show that this is no longer true if we go to non\nb-compact spaces and non\nb-compact quantum groups, by exhibiting a large class of examples of locally compact quantum groups coming from bicrossed product construction, including non\nb-Kac ones, which can faithfully and ergodically act on classical (non\nb-compact) spaces. However, none of these actions can be isometric in the sense of \cite{Goswami:Qnt_isometry}, leading to the conjecture that the result obtained by Goswami and Joardar in~\cite{Goswami-Joardar:Rigidity_CQG_act} about non\nb-existence of genuine quantum isometry of classical compact connected Riemannian manifolds may hold in the non\nb-compact case as well.
\end{abstract}


\thanks{The first author was partially supported by Swarnajayanti grant given by 
  D.S.T., Government of India. The second author was supported by the visiting scientist 
  program at the Indian Statistical Institute, Kolkata. Parts of the manuscript had been 
  written and revised while the second author had been supported by Fields--Ontario 
  postdoctoral fellowship, NSERC and ERA at the University of Ottawa and Carleton 
  University.}

\maketitle

\section{Introduction}
\label{sec:introduction}
Symmetry plays a crucial role in many areas of mathematics and physics. Conventionally, group actions are used to model symmetries and with the advent of more general mathematical structures called quantum groups in~\cites{Drinfeld:Quantum_groups, Jimbo:Yang-baxter_eq, Woronowicz:CQG}. It should be natural to consider the actions (defined in a suitable sense depending on the  algebraic or analytic framework chosen) of quantum groups on classical and noncommutative spaces. In this context, a very interesting programme is to study quantum symmetries of classical spaces. One may hope that there are many more quantum symmetries of a given classical space than classical group symmetries which will help one understand the space better. Indeed, it has been a remarkable discovery of S. Wang that for \(n\geq 4\), a finite set of cardinality \(n\) has an infinite dimensional compact quantum group (`quantum permutation group') of symmetries. For the relevance of quantum group symmetries in a wider and more geometric context, we refer the reader to the discussion on `hidden symmetry in algebraic geometry' in 
\cite{Manin:Qnt_grp_NCG}*{Chapter 13} where Manin made a remark about possible genuine Hopf algebra symmetries of classical smooth algebraic varieties.

Recently,  several examples of  faithful continuous action by genuine (i.e. not of the form \(\Cont(G)\) for a compact group \(G\)) compact quantum groups on \(\Cont(X)\) for a connected compact space \(X\) were constructed by H. Huang~\cite {Huang:Faithful_act_CQG}. In \cite{Etingof-Walton:Semisimple_Hopf_act}  an example of a faithful action by the finite dimensional genuine compact quantum group on the algebra of 
regular function of a non\nb-smooth variety was given. However, it turns out rather formidable to construct such actions when the space is smooth (and connected) and the action is also smooth in some natural sense. In \cite{Goswami-Joardar:Rigidity_CQG_act}, it is conjectured that no smooth faithful action of genuine compact quantum group on a compact connected smooth  manifold can exist. The conjecture has been proved in that paper in two important cases: (i) when the action is isometric and (ii) when the compact quantum group is finite dimensional. We also mention the work of Etingof and Walton 
\cite{Etingof-Walton:Semisimple_Hopf_act} which gives a similar no go result in the algebraic framework.
   
It is, however, expected that genuine non\nb-compact quantum groups may have faithful smooth actions on smooth connected manifolds. Indeed, there are examples of such actions in the algebraic set-up (see~\cite{Goswami-Joardar:Rigidity_CQG_act}*{Example 14.2}). This motivates one to construct examples of faithful $\Cst$\nb-actions of non\nb-compact locally compact quantum groups on \(\Contvin(X)\) where \(X\) is a smooth connected manifold. It is also desirable to see if the actions are smooth in any suitable sense. We give a method to construct such actions using the theory of bicrossed products due to Baaj, Skandalis and Vaes~\cite{Baaj-Skandalis-Vaes:Non-semi-regular} and Vaes and Vainermann \cite{Vaes-Vainerman:Extension_of_lcqg}. There is one remarkable observation: there are several non\nb-Kac quantum groups with faithful \(\Cst\)\nb-actions (even ergodic) on \(\Contvin(X)\). This could not be possible in the realm of compact quantum groups, as any compact quantum group acting faithfully on a commutative \(\Cst\)\nb-
algebra must be of Kac type (see~\cite{Huang:Inv_subset_CQG_act}). Thus, there seems to be more freedom for getting quantum symmetries on classical spaces 
in the realm of locally compact quantum groups than the compact ones. On the other hand, it should be noted that if we are interested in actions which are isometric in a natural sense (as in~\cite{Goswami:Qnt_isometry}). We cannot possibly get any genuine locally compact quantum group actions on classical (connected) Riemannian manifolds. Indeed, such a no-go result is obtained within the class of locally compact quantum groups considered by us in this paper (see Theorem~\ref{the:rig_iso_faithful}).

\medskip
        
The plan of the paper is as follows. We gather some basic definitions and facts about locally compact quantum groups and their actions in the von Neumann as well as \(\Cst\)\nb-algebraic set-up in Subsection~\ref{subsec:LCQG}, followed by a brief account of the bicrossed product construction of locally compact groups in Subsection~\ref{subsec:Bicross}. In Section~\ref{sec:Podles_cond}, we specialise to the locally compact quantum groups arising from bicrossed product construction of two groups say \(G_{1},G_{2}\) forming a matched pair, with \(G_1\) being an abelian Lie group. We describe a natural action of this bicrossed product quantum group  on \(\Contvin(\widehat{G_1})\) and verify that it satisfies Podle\'s\nb-type density conditions. Section~\ref{sec:properties} is devoted to investigate necessary and sufficient conditions for this action to be faithful or isometric. Using this, we observe that that there is a large class of  genuine locally compact quantum groups having faithful actions on commutative \(\Cst\)\nb-algebra of \(\Contvin\)\nb-functions on a locally compact manifold in Section~\ref{sec:Example}. However, it is shown in Subsection~\ref{sec:Isometry} that no such genuine quantum group actions can be isometric. 
    
\section{Preliminaries}
\label{sec:preliminaries}
  All Hilbert spaces and \(\Cst\)\nb-algebras are assumed to be separable. 
  
  For two norm\nb-closed subsets~\(X\) and~\(Y\) of a~\(\Cst\)\nb-algebra, 
    let 
    \[
       X\cdot Y\defeq\{xy : x\in X, y\in Y\}^{\textup{CLS}},
    \]
  where CLS stands for the~\emph{closed linear span}.    
     
  For a~\(\Cst\)\nb-algebra~\(A\), let~\(\Mult(A)\) be its multiplier algebra and 
  \(\U(A)\) be the group of unitary multipliers of~\(A\). The unit of 
  \(\Mult(A)\) is denoted by~\(1_{A}\). 
   Next recall some standard facts about multipliers and morphisms of 
   \(\Cst\)\nb-algebras from~\cite{Masuda-Nakagami-Woronowicz:C_star_alg_qgrp}*{Appendix A}.
   Let~\(A\) and~\(B\) be~\(\Cst\)\nb-algebras. A \Star{}homomorphism 
  \(\varphi\colon A\to\Mult(B)\) is called \emph{nondegenerate} if 
  \(\varphi(A)\cdot B=B\). Each nondegenerate \Star{}homomorphism 
  \(\varphi\colon A\to\Mult(B)\) extends uniquely to a unital 
  \Star{}homomorphism \(\widetilde{\varphi}\) from~\(\Mult(A)\) to 
  \(\Mult(B)\). Let \(\Cstcat\) be the category of \(\Cst\)\nb-algebras with nondegenerate 
  \Star{}homomorphisms \(A\to\Mult(B)\) as morphisms \(A\to B\); let Mor(A,B) denote this 
  set of morphisms. We use the same symbol for an element of~\(\Mor(A,B)\) and its 
  unique extension from~\(\Mult(A)\) to~\(\Mult(B)\).
  
  A \emph{representation} of a
  \(\Cst\)\nb-algebra~\(A\) on a Hilbert space~\(\Hils\) is a 
  nondegenerate \Star{}homomorphism \(\pi\colon A\to\Bound(\Hils)\). 
  Since \(\Bound(\Hils)=\Mult(\Comp(\Hils))\), the nondegeneracy
  conditions \(\pi(A)\cdot\Comp(\Hils)=\Comp(\Hils)\) is equivalent 
  to begin \(\pi(A)(\Hils)\) is norm dense in~\(\Hils\),  
  and hence this is same as having a morphism 
  from~\(A\) to~\(\Comp(\Hils)\). The identity representation 
  of~\(\Comp(\Hils)\) on \(\Hils\) is denoted by~\(\Id_{\Hils}\).
  The group of unitary operators on a Hilbert space~\(\Hils\) 
  is denoted by \(\U(\Hils)\). The identity element in 
  \(\U(\Hils)\) is denoted by~\(1_{\Hils}\).
  
  We use~\(\otimes\) both for the tensor product of Hilbert spaces,  
  minimal tensor product of \(\Cst\)\nb-algebras, and von Neumann 
  algebras which is well understood from the context.
  We write~\(\Flip\) for the tensor flip \(\Hils\otimes\Hils[K]\to
  \Hils[K]\otimes\Hils\), \(x\otimes y\mapsto y\otimes x\), for two 
  Hilbert spaces \(\Hils\) and~\(\Hils[K]\).  We write~\(\flip\) for the
  tensor flip isomorphism \(A\otimes B\to B\otimes A\) for two
  \(\Cst\)\nb-algebras or von Neumann algebras \(A\) and~\(B\). 
  
 Let~\(A_{1}\), \(A_{2}\), \(A_{3}\) be \(\Cst\)\nb-algebras. 
 For any~\(t\in\Mult(A_{1}\otimes A_{2})\) we denote the 
 leg numberings on the level of~\(\Cst\)\nb-algebras as 
 \(t_{12}\defeq t\otimes 1_{A_{3}} \in\Mult(A_{1}\otimes A_{2}\otimes A_{3})\), 
 \(t_{23}\defeq1_{A_{3}}\otimes t_{12}\in\Mult(A_{3}\otimes A_{1}\otimes A_{2})\)
 and~\(t_{13}\defeq\flip_{12}(t_{23})=\flip_{23}(t_{12})\in\Mult(A_{1}\otimes A_{3}\otimes A_{2})\). 
 In particular,  let \(A_{i}=\Bound(\Hils_{i})\) for some Hilbert spaces~\(\Hils_{i}\), where 
 \(i=1,2,3\). Then for any \(t\in\Bound(\Hils_{1}\otimes\Hils_{2})\) 
 the leg numberings are obtained by replacing~\(\flip\) with the 
 conjugation by~\(\Flip\) operator.  
 \subsection{Locally compact quantum groups and their actions}
  \label{subsec:LCQG}
   For a general theory of \(\Cst\)\nb-algebraic 
   locally compact quantum groups we refer~\cites{Masuda-Nakagami-Woronowicz:C_star_alg_qgrp, 
   Kustermans-Vaes:LCQGvN}. 
      
    A~\(\Cst\)\nb-\emph{bialgebra}~\(\Bialg{A}\) is a~\(\Cst\)\nb-algebra 
    \(A\) and a comultiplication~\(\Comult[A]\in\Mor(A,A\otimes A)\)
    that is coassociative:
    \((\Id_{A}\otimes\Comult[A])\circ\Comult[A]=(\Comult[A]\otimes\Id_{A})\circ\Comult[A]\).
    Moreover, \(\Bialg{A}\) is~\emph{bisimplifiable} \(\Cst\)\nb-bialgebra 
    if~\(\Comult[A]\) satisfies the cancellation property 
    \begin{equation}
     \label{eq:cancellation}
      \Comult[A](A)\cdot (1_{A}\otimes A)=\Comult[A](A)\cdot (A\otimes 1_{A})=A\otimes A .
    \end{equation}
   Let~\(\varphi\) be a faithful (approximate) KMS weight
   on~\(A\)    (see~\cite{Kustermans-Vaes:LCQG}*{Section 1}). 
   The set of all positive 
   \(\varphi\)\nb-integrable and~\(\varphi\)\nb-square integrable 
   elements are defined by~\(\mathcal{M}_{\varphi}^{+}
   \defeq\{a\in A^{+}\text{ \(\mid\) \(\varphi(a)<\infty\)}\}\) and 
   \(\mathcal{N}_{\varphi}\defeq\{a\in A\text{ \(\mid\) \(\varphi(a^{*}a)<\infty\)}\}\), 
   respectively. Moreover, \(\varphi\) is  called
   \begin{enumerate} 
   \item \emph{left invariant} if     
   \(\omega((\Id_{A}\otimes\varphi)\Comult[A](a))=\omega(1)\varphi(a)\)
     for all~\(\omega\in A_{*}^{+}\), \(a\in\mathcal{M}^{+}_{\varphi}\);
   \item \emph{right invariant} if       
   \(\omega((\varphi\otimes\Id_{A})\Comult[A](a))=\omega(1)\varphi(a)\)
   for all~\(\omega\in A_{*}^{+}\), \(a\in\mathcal{M}^{+}_{\varphi}\).
   \end{enumerate}
   \begin{definition}[\cite{Kustermans-Vaes:LCQG}*{Definition 4.1}]
    \label{def:Qnt_grp}
    A \emph{locally compact quantum group} 
    (\emph{quantum groups} from now onwards) is 
    a bisimplifiable~\(\Cst\)\nb-bialgebra~\(\Qgrp{G}{A}\) with left
    and right invariant approximate KMS weights~\(\varphi\) and~\(\psi\), 
    respectively. 
   \end{definition}
   By~\cite{Kustermans-Vaes:LCQG}*{Theorem 7.14 \& 7.15}, 
   invariant weights~\(\varphi\) and~\(\psi\) are unique up to a 
   positive scalar factor; 
   hence they are called the left and right 
   \emph{Haar weights} for \(\G\). Moreover,  
   there is a unique (up to isomorphism) Pontrjagin 
   dual~\(\DuQgrp{G}{A}\) of~\(\G\), which is again a quantum group. 
   
   Next we consider the GNS triple~\((\textup{L}^{2}(\G),\pi,\Lambda)\) 
   for~\(\varphi\). There is an element \(\Multunit\in\U(\textup{L}^{2}(\G)\otimes\textup{L}^{2}(\G))\) 
   satisfying the pentagon equation:
     \begin{equation}
      \label{eq:pentagon}
      \Multunit_{23}\Multunit_{12}
     = \Multunit_{12}\Multunit_{13}\Multunit_{23}
     \qquad
     \text{in \(\U(\textup{L}^{2}(\G)\otimes\textup{L}^{2}(\G)\otimes\textup{L}^{2}(\G)).\)}
   \end{equation}
   \(\Multunit\) is called \emph{multiplicative unitary}.
   Furthermore,~\cite{Kustermans-Vaes:LCQG}*{Proposition 6.10}  shows that \(\Multunit\) 
   is manageable (in the sense of~\cite{Woronowicz:Multiplicative_Unitaries_to_Quantum_grp}*{Definition 1.2}). 
   Also~\(\Multunit\) \emph{generates}~\(\G\) (or~\(\G\) is \emph{generated} by~\(\Multunit\)) in the 
   following sense:
   \begin{enumerate}
   \item the dual multiplicative unitary 
   \(\DuMultunit\defeq\Flip\Multunit[*]\Flip\in\U(\textup{L}^{2}(\G)\otimes\textup{L}^{2}(\G))\) 
   is also manageable.
    \item the slices of~\(\Multunit\) defined by 
    \begin{alignat}{2}
     \label{eq:slice_first}
     \hat{A} &\defeq \{(\omega\otimes\Id_{\textup{L}^{2}(\G)})\Multunit :
    \omega\in\Bound(\textup{L}^{2}(\G))_*\}^\CLS ,\\
    \label{eq:slice_second}
      A &\defeq\{(\Id_{\textup{L}^{2}(\G)}\otimes\omega)\Multunit :
    \omega\in\Bound(\textup{L}^{2}(\G))_*\}^\CLS ,
   \end{alignat} 
   are nondegenerate \(\Cst\)\nb-subalgebras of \(\Bound(\textup{L}^{2}(\G))\). 
    \item  \(\Multunit\in\U(A\otimes\hat{A})\subseteq\U(\textup{L}^{2}(\G)\otimes
  \textup{L}^{2}(\G))\). 
    \item the comultiplication maps~\(\Comult[A]\) and 
  \(\DuComult[A]\) are defined by  
  \begin{equation}
   \label{eq:comults}  
   \Comult[A](a)\defeq\Multunit[*](1\otimes a)\Multunit ,
   \qquad
    \DuComult[A](\hat{a})\defeq\flip\big(\Multunit(\hat{a}\otimes 1)\Multunit[*]\big),
  \end{equation}
   for all~\(a\in A\), \(\hat{a}\in\hat{A}\).   
\end{enumerate}


A general theory of locally compact quantum groups in the von\nb-Neumann algebraic framework has been developed by Kustermans and Vaes in~\cite{Kustermans-Vaes:LCQGvN}. Moreover, there is a nice interplay between 
\(\Cst\)\nb-algebraic and von\nb-Neumann algebraic locally compact quantum groups in general via multiplicative unitaries. We briefly recall this in the group case.

 Let~\(G\) be a locally compact group and~\(\mu\) be its left Haar measure. Define, 
 \(\Multunit\in\U(L^{2}(G\times G,\mu\times \mu)\) by \(\Multunit\xi(x,y)=\xi(x,x^{-1}y)\) for all 
 \(\xi\in L^{2}(G\times G,\mu\times \mu)\), \(x,y\in G\). 
 A simple computation shows that~\(\Multunit\) is a multiplicative 
 unitary. Furthermore, 
 \begin{alignat*}{2}
    A &=\Contvin(G)=\{(\Id_{\textup{L}^{2}(G,\mu)}\otimes\omega)\Multunit :
    \omega\in\Bound(\textup{L}^{2}(G))_*\}^\CLS ,\\
    M &=L^\infty(G)=\{(\Id_{\textup{L}^{2}(\G)}\otimes\omega)\Multunit :
    \omega\in\Bound(\textup{L}^{2}(\G))_*\}^{\textup{weak closure}}.
 \end{alignat*}
 Using~\eqref{eq:comults} we get~\(\Comult[M]\colon L^\infty(G)\to L^\infty(G\times G)\) by 
 \(\Comult[M](f)(x,y)=f(xy)\) for all~\(f\in L^\infty(G)\). The pair~\((M,\Comult[M])\) is a 
 von\nb-Neumann bialgebra. The left Haar weight \(\varphi\) on~\(L^{\infty}(G)\) 
 is given by the integration with respect to the left Haar measure~\(\mu\) on~\(G\): 
\(\varphi(f)\defeq\int f(h) \diff\mu (h)\) for~\(f\in L^\infty(G_{2})^{+}\). Similarly, the 
right Haar weight~\(\psi\) on~\((M,\Comult[M])\) is obtained from the to the right Haar measure on 
\(G\). The pair \(\Qgrp{G}{M}\) is the von Neumann algebraic locally compact quantum group 
associated to~\(G\).

 Similarly,~\(\Comult[A]\colon \Contvin(G)\to\Contb(G\times G)\) is obtained 
 by restricting~\(\Comult[M]\) on~\(\Contvin(G)\). The restriction of left and right Haar weight 
 of~\((M,\Comult[M])\) on~\(\Contvin(G)\) defines the left and Haar weight on~\((A,\Comult[A])\), respectively.
Thus~\((A,\Comult[A])\) is the \(\Cst\)\nb-algebraic version of~\(\G\). 

 Let~\(\hat{A}=\Cred(G)\) and~\(\hat{M}=L(G)\). Let~\(\lambda\) be the left regular representation of 
 \(G\) on~\(L^{2}(G,\mu)\). Define, \(\DuComult[M]\colon L(G)\to L(G\times G)\) by~\(\DuComult(\lambda_g)=\lambda_g
 \otimes\lambda_g\) and \(\DuComult[A]\colon \Cred(G)\to\Mult(\Cred(G\times G))\) as the restriction of~\(\DuComult[M]\). 
 The left and right invariant invariant Haar weights on~\(\hat{M}\) is given by~\(\hat{\varphi}(\lambda(f))=f({1_{G}})\) for 
 all~\(f\in\CompSupp(G)\) such that~\(\lambda(f)\in L(G)^{+}\). Then \(\DuQgrp{G}{A}\) and~\(\DuQgrp{G}{M}\) 
 are the dual of~\(\G\) in~\(\Cst\)\nb-algebraic and von Neumann algebraic setting, respectively. 

 \begin{definition}
  \label{def:cont_action}
  A \emph{\textup(right\textup) \(\Cst\)\nb-action} of~\(\G\) 
  on a \(\Cst\)\nb-algebra~\(C\) is a morphism \(\gamma\colon C\to C\otimes
  A\) with the following properties:
  \begin{enumerate}
  \item \(\gamma\) is a comodule structure, that is,
    \begin{equation}
      \label{eq:right_action}
       (\Id_{C}\otimes\Comult[A])\gamma=(\gamma\otimes\Id_{A})\gamma ;
    \end{equation}
  \item \(\gamma\) satisfies the \emph{Podleś condition}:
  \begin{equation}
   \label{eq:Podles_cond}
    \gamma(C)\cdot(1_C\otimes A)=C\otimes A .
   \end{equation}
  \end{enumerate}
 \end{definition}    
 Some authors demands~\(\gamma\) to be injective. Similarly, the von Neumann algebraic 
 version~\(\Bialg{M}\) of \(\G\) acts on von Neumann algebras 
 as well.
 \begin{definition}
  \label{def:vN_coact}
    A \emph{\textup(right\textup) von Neumann algebraic action} of~\(\G\) on a 
    von Neumann algebra~\(N\) is a faithful, normal, unital \Star{}homomorphism 
 \(\gamma\colon N\to N\otimes M\) satisfying~\((\Id_{N}\otimes\Comult[M])\gamma=(\gamma\otimes\Id_{M})\gamma\).
 \end{definition}
 
\subsection{Bicrossed product of groups}
 \label{subsec:Bicross}
 Bicrossed product construction for mathced pair of locally 
 compact quantum groups goes back to the work of Baaj and Vaes 
 \cite{Baaj-Vaes:Double_cros_prod} and Vaes and 
 Vainerman~\cite{Vaes-Vainerman:Extension_of_lcqg}. 
 In this article, we shall restrict out attention to the bicrossed product 
 construction for locally compact groups. 
 \begin{definition}[\cite{Vaes-Vainerman:Extension_of_lcqg}*{Definition 4.7}]
  \label{def:mathced_pair}
  Let~\(G_{1}\), \(G_{2}\) and~\(G\) be locally compact groups with fixed 
  left Haar measures. The pair \((G_{1},G_{2})\) is called a \emph{matched pair} 
  if
 \begin{enumerate}
  \item there exist a homomorphism \(i\colon G_{1}\hookrightarrow G\) and 
  an anti\nb-homomorphism \(j\colon G_{2}\hookrightarrow G\) with closed 
  images and homeomorphism onto these images;
  \item \(\theta\colon G_{1}\times G_{2}\to G\) defined by 
           \(\theta((g,h))=i(g)j(h)\) is an homeomorphism 
           onto an open subgroup~\(\Omega\) of~\(G\) having a 
           complement of measure zero. 
 \end{enumerate}
\end{definition}
 This allows to define almost everywhere and measurable 
 left action \((\alpha_{g})_{g\in G_{1}}\) of~\(G_{1}\) 
 on~\(G_{2}\) and right action \((\beta_{h})_{h\in G_{2}}\) on 
 \(G\), and satisfies \(j(\alpha_{g}(h))i(\beta_{h}(g))=i(g)j(h)\) 
 for almost all~\(g\in G_{1}\) and~\(h\in G_{2}\). 
 
 By~\cite{Vaes-Vainerman:Extension_of_lcqg}*{Lemma 4.9}, 
 the maps \(G_{1}\times G_{2}\to G_{2}\colon (g,h)\mapsto\alpha_{g}(h)\) and 
 \(G_{1}\times G_{2}\to G_{1}\colon (g,h)\mapsto\beta_{h}(g)\) 
 are measurable, defined almost everywhere, 
 and satisfy the following relations:
 \begin{alignat}{2}
  \label{eq:alpha_comp}
  \alpha_{gs}(h) &=\alpha_{g}(\alpha_{s}(h)), 
  \qquad \beta_{h}(gs) &=\beta_{\alpha_{s}(h)}(g)\beta_{h}(s),\\
  \label{eq:beta_comp}  
 \beta_{ht}(g) &=\beta_{h}(\beta_{t}(g)),
 \qquad  \alpha_{g}(ht) &=\alpha_{\beta_{t}(g)}(h)\alpha_{g}(t).
\end{alignat} 
 for almost all~\(g, s\in G_{1}\), \(h, t\in G_{2}\). Also, \(\alpha_{g}(1_{G_{2}})=1_{G_{2}}\), 
 and \(\beta_{h}(1_{G_{1}})=1_{G_{1}}\) for all~\(g\in G\), \(h\in H\).

Then \(\alpha\colon L^{\infty}(G_{2})\to L^{\infty}(G_{1}\times G_{2})\) 
defined by \(\alpha(f)(g,s)\defeq f(\alpha_{g}(s))\) is a (left) von Neumann algebraic 
action of the locally compact quantum group \(\G_{1}=(\Bialg{L^{\infty}(G_{1})}\)  
on~\(L^{\infty}(G_{2})\). The von Neumann algebraic version of the 
bicrossed product~\(\Qgrp{G}{M}\) is given by
\begin{equation}
 \label{eq:bicros_qnt_grp}
 M\defeq \bigl(\alpha(L^{\infty}(G_{2}))(L(G_{1})\otimes 1)\bigr)'' ,
 \qquad
 \Comult[M](z)\defeq\Multunit[*](1\otimes z)\Multunit 
\end{equation} 
 Let~\(\lambda\) be the left regular representation of~\(G_{1}\). 
 The left (and right) Haar weight on \(L^{\infty}(G_{2})\). 
 Let~\(\hat{\alpha}\) be the dual action of~\(L(G_{1})\) on the 
crossed product~\(M\). Then~\(\hat{\alpha}\colon M\to L(G_{1})\otimes M\) 
is defined by~\(\hat{\alpha}(\alpha(y))=1_{L(G_{1})}\otimes\alpha(\eta)\) 
and \(\hat{\alpha}(x\otimes 1_{L^\infty(G_2)})=\DuComult[L(G_{1})](x)\otimes 
1_{L^\infty(G_2)}\) for all~\(x\in L(G_{1})\) and~\(y\in L^{\infty}(G_2)\). 

Let~\(\hat{\varphi}_{1}\) and~\(\varphi_{2}\) be the left Haar weights on 
\(L(G_{1})\) and~\(L^{\infty}(G_{2})\), respectively. 
By~\cite{Vaes-Vainerman:Extension_of_lcqg}*{Definition 1.13}, the 
left Haar weight~\(\varphi\) on~\(\G\) is given by 
\(
\varphi\defeq \varphi_{2}\alpha^{-1}(\hat{\varphi}_{1}\otimes\Id\otimes\Id)\hat{\alpha}
\).
A simple computation shows, for any~\(f\in\mathcal{N}_{\hat{\varphi}_{1}}\) and 
\(\eta\in\mathcal{N}_{\varphi_{2}}\), 
\begin{equation}
 \label{eq:Haar_wt_prod}
 \varphi\bigl((\alpha(\eta)(\lambda(f)\otimes 1)\bigr)=\hat{\varphi}_{1}(\lambda(f))\varphi_{2}(\eta) 
 =\varphi\bigl((\lambda(f)\otimes 1)(\alpha(\eta)\bigr)
\end{equation}  

Let~\(\Hils=L^{2}(G_{1}\times G_{2})\) be the Hilbert space of square integrable 
functions with respect to the product of the left Haar measures of~\(G_{1}\) and~\(G_{2}\). 
Using \cite{Baaj-Skandalis-Vaes:Non-semi-regular}*{Definition 3.3} for left Haar measure, we 
obtain a multiplicative unitary~\(\Multunit\in\U(\Hils\otimes\Hils)\) for~\(\G\) 
defined by 
\begin{equation}
 \label{eq:Multunit_bicros}
 \Multunit\xi(g,s,h,t)\defeq\xi(\beta_{\alpha_{g}(s)^{-1}t}(h)g,s,h,\alpha_{g}(s)^{-1}t),
\end{equation}
for~\(\xi\in L^{2}(G_{1}\times G_{2}\times G_{1}\times G_{2})\),  
and for almost all~\(g, s \in G_{1}\), \(h, t\in G_{2}\).

Finally, we recall the \(\Cst\)\nb-algebraic version of~\(\G\) from  
\cite{Baaj-Skandalis-Vaes:Non-semi-regular}*{Section 3}. 
Equip the quotient space~\(G_{1}\backslash G\) 
with its canonical invariant measure class. Then 
embedding \(G_{2}\to G_{1}\backslash G\) identifies \(G_{2}\) 
with a Borel subset of \(G_{1}\backslash G\) with complement of measure zero. 
Then~\cite{Baaj-Skandalis-Vaes:Non-semi-regular}*{Proposition 3.2} gives an 
isomorphism between \(L^{\infty}(G_{2})\) and \(L^{\infty}(G_{1}\backslash G)\); 
hence we can restrict~\(\alpha\) to \(\Contvin(G_{1}\backslash G)\). 
The \(\Cst\)\nb-algebraic version~\(\Bialg{A}\) of~\(\G\) is given 
by~\cite{Baaj-Skandalis-Vaes:Non-semi-regular}*{Proposition 3.6}  
\begin{equation}
 \label{eq:Cst_bicros}
 A\defeq \alpha(\Contvin(G_{1}\backslash G))\cdot (\lambda(\Cred(G_{1}))\otimes 1),
 \qquad
 \alpha\in\Mor(\Contvin(G_{2}), A) . 
\end{equation}
The dual~\(\DuQgrp{G}{A}\) is obtained by exchanging the 
roles of~\(G_{1}\) and~\(\alpha\) by \(G_{2}\) and \(\beta\), respectively. 
\section{Existence of C*-actions of bicrossed products on spaces}
 \label{sec:Podles_cond}
 Let~\(G_{1}\), \(G_{2}\) be locally compact groups and 
 \((G_{1},G_{2})\) form a matched pairs. Let~\(\Qgrp{G}{A}\) 
 be the associated (\(\Cst\)\nb-algebraic) bicrossed product 
 quantum group. From now onwards we assume that~\(G_1\) is abelian. 
\begin{theorem}
 \label{the:Cst_coact}
There is a \(\Cst\)\nb-action of~\(\G\) on~\(\Contvin(\widehat{G_{1}})\). 
\end{theorem}
Throughout~\(\lambda\) denotes the (right) regular representation of~\(G_{1}\): 
\(\lambda(f)\xi'(h)\defeq\int f(g)\xi'(hg) \diff\mu_{1}(g)\), 
for~\(f\in\CompSupp(G_{1})\) and \(\xi'\in L^{2}(G_{1})\).

Clearly, we have an element~\(i\in\Mor(\Cred(G_1),A)\) and its extension, 
denoted by~\(i\) again, to~\(\Mult(\Cred(G_1))\) is given by 
\(i(x)\defeq x\otimes 1\) for all~\(x\in\Mult(\Cred(G_{1}))\). 
Hence, \(\gamma\defeq\Comult[A]\circ i\) is an element 
in \(\Mor(\Cred(G_{1}),A\otimes A)\). In order to interpret \(\gamma\) 
as the desired \(\Cst\)\nb-action in Theorem~\ref{the:Cst_coact} 
following proposition will be crucial.
\begin{proposition}
 \label{prop:def_gamma}
 \(\gamma\) is a \Star{}homomorphism from~\(\Cred(G_{1})\) to~\(\Mult(\Cred(G_{1})\otimes A)\) and 
 defined by 
 \begin{equation}
  \label{eq:def_gamma}
  \gamma(\lambda(f))\xi(g,h,t)\defeq \int f(z)\xi(g\beta_{\alpha_{h}(t)}(z),hz,t) \diff\mu_{1}(z)
 \end{equation}
 for all~\(f\in\CompSupp(G_{1})\) and~\(\xi\in L^{2}(G_{1}\times G_{1}\times G_{2})\). 
\end{proposition}  
\begin{proof}
 The adjoint~\(\Multunit[*]\in\U(L^{2}(G_{1}\times G_{2}\times G_{1}\times G_{2}))\) of 
 the multiplicative unitary~\(\Multunit\) of~\(\G\) in~\eqref{eq:Multunit_bicros} is defined 
 by 
 \[
    \Multunit[*]\xi'(g,s,h,t)\defeq\xi'(g',s,h,\alpha_{g'}(s)t)
    \qquad\text{for~\(\xi'\in L^{2}(G_{1}\times G_{2}\times G_{1}\times G_{2})\),}
 \]    
 where~\(g'\defeq\beta_{t}(h)^{-1}g\) for all~\(g,h\in G_{1}\) and~\(s,t\in G_{2}\). 
 Using the definition of the comultiplication~\ref{eq:comults} and 
 the property~\eqref{eq:beta_comp} of \(\beta\) we obtain
 \begin{align}
  \label{eq:gamma}
 \gamma(\lambda(f)) \xi'(g,s,h,t) \nonumber
 &= \Comult[A](1_{L^{2}(G_{1}\times G_{2})}\otimes\lambda(f)\otimes 1_{L^{2}(G_{2})})\xi'(g,s,h,t)\\ \nonumber
 &=\bigl(\Multunit[*](1_{L^{2}(G_{1}\times G_{2})}\otimes\lambda(f)\otimes 1_{L^{2}(G_{2})})\Multunit\bigr)\xi'(g,s,h,t)\\ \nonumber
 &= \bigl((1_{L^{2}(G_{1}\times G_{2})}\otimes\lambda(f)\otimes 1_{L^{2}(G_{2})})\Multunit\bigr)\xi'(g',s,h,\alpha_{g'}(s)t)\\ \nonumber 
 &= \Multunit\Bigl(\int f(z)\xi'(g',s,hz,\alpha_{g'}(s)t)\diff\mu_{1}(z)\Bigr)\\ 
 &=\int f(z)\xi'(g\beta_{\alpha_{h}(t)}(z),s,hz,t)\diff\mu_{1}(z).
\end{align} 
Since \(G_{1}\) is abelian, we indenify~\(\Mult(\Cred(G_{1})\otimes A)\) with 
\(\Contb(\widehat{G_{1}},\Mult(A))\subset\Bound(L^{2}(\widehat{G_{1}}\times G_{1}\times G_{2}))\) using 
the Fourier transform. 

Let~\(f\colon \widehat{G_1}\to\Mult(A)\) be a strictly 
continuous function. Define an operator~\(M_{f}\) acting on 
\(L^{2}(\widehat{G_1}\times G_{1}\times G_{2})\) by 
\[
  M_{f}(\xi_{1}\otimes\xi_{2})(\hat{g},h,t) 
  \defeq \xi_{1}(\hat{g}) (f(\hat{g})\xi_{2})(h,t) 
  \quad\text{for all~\(\xi_{1}\in L^{2}(\widehat{G_1})\), 
  \(\xi_{2}\in L^{2}(G_{1}\times G_{2})\),}
\]
which is an element in 
\(M_{f}\in\Contb(\widehat{G_1},\Mult(A))\cong\Mult(\Cred(G_{1})\otimes A)\). 

Recall the dual pairing~\(\langle\cdot,\cdot\rangle\colon G_{1}\times\widehat{G_1}\to \mathbb{T}\) 
defined by~\(\langle h,\hat{g}\rangle\defeq \hat{g}(h)\) for all 
\(h\in G_{1}\) and~\(\hat{g}\in\widehat{G_1}\). Let~\(D\) be a~\(\Cst\)\nb-algebra. 
For an element \(F\in L^{1}(G_{1},D)\), the Fourier transform~\(\widehat{F}\) is  defined by 
\(\widehat{F}(\hat{g})\defeq\int F(h)\langle h, \hat{g}\rangle \diff\mu_{1}(h)\) for 
\(\hat{g}\in\widehat{G_{1}}\), and \(\widehat{F}\) is an element in~\(\Contvin(\widehat{G_{1}},D)\).

Then, for a fixed~\(z\in G_{1}\), consider~\(T_{z}\colon\widehat{G_{1}}\to\Mult(A)\) 
defined by 
\[
 T_{z}(\hat{g})\xi_{2}(h,t)\defeq \langle 
\beta_{\alpha_{h}(t)}(z),\hat{g}\rangle\xi_{2}(hz,t)  
\qquad 
\text{for all~\(\xi_{2}\in L^{2}(G_{1}\times G_{2})\).}
\] 
Clearly, \(\hat{g}\mapsto T_{z}(\hat{g})\) is continuous in the strict 
topology. Hence, any~\(f\in \CompSupp(G_{1})\) gives 
\(\int f(z)T_{z} \diff\mu_{1}(z)\) is an element 
in~\(\Contb(\widehat{G_1},\Mult(A)) \cong\Mult(\Cred(G_{1})\otimes A)\).

Using the Fourier transform in the first leg of~\eqref{eq:gamma} we observe 
\[
 \gamma(\lambda(f))=\Flip_{12} (\int f(z) (\Id_{L^{2}(G_{2})}\otimes T_{z} \diff\mu_{1}(z))
  \Flip_{12}
\]
in~\(\Bound(L^{2}(\widehat{G_1}\times G_{2}\times G_{1}\times G_{2}))\). 
Finally, putting~\(s=1_{G_{2}}\) in~\eqref{eq:gamma} we obtain~\eqref{eq:def_gamma}.
\end{proof}
We gather some standard facts related to Fourier transform in the next lemma.
\begin{lemma}
 \label{lemm:var_prop}
 Let \(K\) be a compact subset of~\(G_{1}\). Define 
 \begin{align*}
   S(K,D) &\defeq \{F\in L^{1}(G_{1},D) \text{ \(\mid\)~\(\textup{supp}(F)\subset K\)}\}\subset L^{1}(G,D);\\ 
   \widehat{S(K,D)} &\defeq\{\widehat{F} \text{ \(\mid\)~\(F\in S(K,D)\)}\}\subset 
 \Contvin(\widehat{G_{1}},D).
 \end{align*}
 We have the following:
\begin{enumerate}
 \item\label{eq:can_der} Let~\(\widehat{G_{1}}\) is a Lie group, 
 and let \(\delta=(\delta_{1},\cdots,\delta_{n})\) be 
 the canonical derivation on \(\widehat{G_{1}}\) for~\(i=1,\cdots , n\). 
 Then \(\widehat{S(K,D)}\subset\dom(\tilde{\delta})\) 
 and \(\tilde{\delta}\bigl(\widehat{S(K,D)}\bigr)\subset \widehat{S(K,D)}\), 
 where~\(\tilde{\delta}_{i}\defeq\delta_{i}\otimes \Id_{D}\) for~\(i=1,\cdots , n\), 
 and~\(\tilde{\delta}\defeq(\tilde{\delta}_{1},\cdots,\tilde{\delta}_{n})\).
 \item\label{eq:isometry} \(S(K,D), \widehat{S(K,D)}\subset L^{2}(G_{1},D)\).
 \item\label{eq:morph} Let~\(D'\) be a~\(\Cst\)\nb-algebra and 
 \(\rho\colon D\to D'\) be a completely bounded map.  Then 
 \((\Id\otimes\rho)S(K,D)\subset S(K,D')\) and 
 \((\Id\otimes\rho)\widehat{S(K,D)}\subset\widehat{S(K,D')}\).
\end{enumerate}
\end{lemma}
\begin{proof}
 By definition~\(\delta_{i}(\lambda_{g})\defeq\rho_{i}(g)\lambda_{g}\), where 
 \(\lambda_{g}\in\Mult(\Cred(G_{1}))\cong\Contb(\widehat{G_{1}})\) 
 and~\(\rho_{i}\in\widehat{G_{1}}\) for~\(i=1,\cdots, n\); hence 
 gives~\ref{eq:can_der}. The second fact follows because the 
 Fourier transform is \(L^{2}\)\nb-isometry. The last fact is 
 trivial.
\end{proof}
 For any~\(f\in\CompSupp(G_{1})\) and~\(\eta\in\CompSupp(G_{2})\) 
 define~\(\hat{\pi}_{1}(f)\defeq \lambda(f)\otimes 1\) and~\(\pi_{2}(\eta)
 \defeq\alpha(\eta)\). Therefore, \(\hat{\pi}_{1}(f)\pi_{2}(\eta)\in A\). 
 \begin{proposition}
  \label{prop:str_podles}
  Define \(w\defeq\gamma(\hat{\pi}_{1}(f))(1\otimes\pi_{2}(\eta))
 \in S(K,\Mult(A))\), where \(K=\textup{supp}(f)\). 
  Then~\(p\mapsto (\Id\otimes\varphi)(\hat{w}(p)^{*}\hat{w}(p))\in\Contvin(\widehat{G_{1}})\). 
 \end{proposition}
 \begin{proof}
   By definition
   \[ 
   w^{*}w
              =(1\otimes\pi_{2}(\eta^{*}))\gamma\bigl(\hat{\pi}_{1}(\abs{f}^{2})\bigr)(1\otimes\pi_{2}(\eta))
   \]           
  Using~\eqref{eq:Haar_wt_prod} we get
  \begin{align*}
       (\Id\otimes\varphi)(w^{*}w)
    &=(\Id\otimes\varphi) 
        \bigl(1\otimes\pi_{2}(\eta^{*}))\gamma(\hat{\pi}_{1}(\abs{f}^{2}))(1\otimes\pi_{2}(\eta)\bigr)\\
    &=(\Id\otimes\varphi)
    \bigl(1\otimes\pi_{2}(\abs{\eta}^{2})\gamma(\hat{\pi}_{1}(\abs{f}^{2})\bigr)        
  \end{align*}  
  By virtue of Proposition~\ref{prop:def_gamma} and~\eqref{eq:Haar_wt_prod}, 
  for~\(p\in\widehat{G_{1}}\), we get
  \begin{align*}
       (\Id\otimes\varphi)(\hat{w}(p)^{*}\hat{w}(p))
    &= \iint \abs{f}^{2}(g)\abs{\eta}^{2}(s)\langle \beta_{s}(g), p\rangle \diff\mu_{1}(g)\diff\mu_{2}(s)\\
    &= \iint \abs{f}^{2}(\beta_{s^{-1}}(g'))\abs{\eta}^{2}(s)\theta(g',s)\langle g',p\rangle \diff\mu_{1}(g')\diff\mu_{2}(s),
  \end{align*}     
 where~\(g'=\beta_{s}(g)\) and~\(\theta(g,s)\defeq\abs{\frac{\diff}{\diff\mu_{1}}\beta_{s^{-1}}(g)}\). 
 
 Define~\(G_{s}(g)\defeq\abs{f}^{2}(\beta_{s^{-1}}(g))\abs{\eta}^{2}(s)\theta(g,s)\). Then, 
 \[
   (\Id\otimes\varphi)(\hat{w}(p)^{*}\hat{w}(p))=\int\widehat{G_{s}}(p)\diff\mu_{2}(s).
 \]  
 A simple computation gives
 \begin{align*}
   \norm{G_{s}}_{1} = \int\norm{G_{s}(g)} \diff\mu_{1}(g) 
                              &=\int\abs{f}^{2}(\beta_{s^{-1}}(g))\abs{\eta}^{2}(s)\theta(g,s) \diff\mu_{1}(g)\\
                              &=\int\abs{f}^{2}(g)\abs{\eta}^{2}(s) \diff\mu_{1}(g)= (\norm{f}_{2})^{2}\abs{\eta}^{2}(s).
 \end{align*}
Therefore,~\(\widehat{G_{s}}\in\Contvin(\widehat{G_{1}})\) for almost all~\(s\in G_{2}\). 
Also~\(\int\abs{\eta}^{2}(s) \diff\mu_{2}(s)\le\infty\). By dominated 
convergence theorem, for any sequence~\(\{p_{n}\}\subset\widehat{G_{1}}\) such that 
\(\abs{p_{n}}\to\infty\) as~\(n\to\infty\), we have   
\[
  \lim_{n\to\infty}\int\widehat{G_{s}}(p_{n})\diff\mu_{2}(s) 
  =\int\bigl(\lim_{n\to\infty}\widehat{G_{s}}(p_{n})\bigr) \diff\mu_{2}(s) 
  =0. \qedhere
\]  
 \end{proof}
\begin{proof}[Proof of Theorem~\textup{\ref{the:Cst_coact}}]
 Proposition~\ref{prop:def_gamma} and Proposition~\ref{prop:str_podles} 
 give \(\gamma\) is an element in 
 \(\Mor(\Contvin(\widehat{G_{1}}),\Contvin(\widehat{G_{1}})\otimes A)\). 
 The coassociativity of~\(\Comult[A]\) gives~\eqref{eq:right_action} 
 for~\(\gamma\). 

Let~\((L^{2}(\G),\pi,\Lambda)\) be the GNS triple for the the left Haar weight 
\(\varphi\) in~\eqref{eq:Haar_wt_prod}. For any~\(v\in L^{2}(\G)\), define 
the operator~\(\Theta_{v}(c)\defeq cv\) for~\(c\in\C\). 
Let~\((e_{i})_{i\in\N}\) be an orthonormal basis of~\(L^{2}(\G)\). 
For~\(w\in S(K,\Mult(A))\) in Proposition~\ref{prop:str_podles}, 
define 
\[
   x_{i}^{*}(p)\defeq (\Id\otimes\Theta_{e_{i}}^{*})(\Id\otimes\Lambda)\Comult[A](\hat{w}(p)) 
   \in\Mult(A) \qquad\text{for~\(p\in\widehat{G_{1}}\).}
\]  
Also, for~\(q\in\mathcal{N}_{\varphi}\) define~\(q_{i}\defeq 
(\Id\otimes\Theta_{e_{i}}^{*})(\Id\otimes\Lambda)\Comult[A](q)\in\Mult(A)\). 

We compute,
\begin{align*}
     \sum_{i=1}^{\infty}x_{i}(p)q_{i}
  &= \sum_{i=1}^{\infty}(\Id\otimes\Lambda)\Comult[A](w(p))^{*}(1\otimes\Theta_{e_{i}}\Theta_{e_{i}}^{*})
                                    (\Id\otimes\Lambda)\Comult[A](q)\\
  &= (\Id\otimes\Lambda)\Comult[A](w(p))^{*}(\Id\otimes\Lambda)\Comult[A](q)\\
  &=\varphi(\hat{w}(p)^{*}q)1_{\Mult(A)}                                   
\end{align*} 
By virtue of Proposition~\ref{prop:str_podles}, 
\(F_{N}(p)\defeq\Sigma_{i=1}^{N}x_{i}(p)q_{i}\) is strictly convergent. 
Hence, for any given~\(q'\in A\), the sequence~\(\{F_{N}(1\otimes q')\}\) 
converges uniformly over every compact subset of~\(\widehat{G_{1}}\). 

In order to establish Podle\'s condition~\eqref{eq:Podles_cond} 
for~\(\gamma\) we need to show \(\{F_{N}(1\otimes q')\}\) 
converges uniformly over~\(\widehat{G_{1}}\) for~\(q'\in A\). 

By a similar argument used in~\cite{Vaes-VanDaele:Hopf_Cstalg}*{Proposition 5.11}, and 
Proposition~\ref{prop:str_podles} gives \(\sum_{i=1}^{n}x_{i}^{*}(p)x_{i}(p)\) 
strictly converges to~\(\varphi(\hat{w}(p)^{*}\hat{w}(p))1_{\Mult(A)}\). Similarly, 
\(\sum_{i=1}^{n}q_{i}^{*}q_{i}\) is bounded and strictly convergent.  
Let~\(\norm{\sum_{i=1}^{n}q_{i}^{*}q_{i}}<C^{2}\). 
Given~\(\epsilon> 0\), we can choose a compact subset 
\(K'\) of~\(\widehat{G_{1}}\), such that~\((\Id\otimes\varphi)(\hat{w}^{*}(p)\hat{w}(p))
\leq(\frac{\epsilon}{C})^{2}\) for all~\(p\notin K'\). 
Hence~\(\norm{\sum_{i=m}^{n}x_{i}(p)x_{i}^{*}(p)}\leq(\frac{\epsilon}{C})^{2}\) for all 
\(p\notin K'\), and for all~\(m, n\). 

Now choose~\(N_{0}\) such that for all~\(m,n\geq N_{0}\), 
\(\norm{(F_{m}-F_{n})(p)q'}<\epsilon\) for~\(p\in K'\), \(q'\in A\). 

Finally, for all~\(m,n\geq N_{0}\) 
\[
  \left\Vert(F_{m}-F_{n})(p)\right\Vert
  \leq\left\Vert\sum_{i=m}^{n}x_{i}(p)x_{i}^{*}(p)\right\Vert^{\frac{1}{2}}
  \left\Vert\sum_{i=m}^{n} q_{i}^{*}q_{i}\right\Vert^{\frac{1}{2}}
  <\epsilon\norm{q'},
\]   
for~\(p\notin K'\). Hence~\(\{F_{N}(1\otimes q')\}\) is 
Cauchy sequence in norm for~\(q'\in A\).
\end{proof}  
\section{Properties of bicrossed product C*-actions} 
 \label{sec:properties}
 Let~\(G_{1}\), \(G_{2}\), and \(\Qgrp{G}{A}\) be as before, so that 
 \(G_{1}\) is abelian. In this section we shall discuss various properties of the 
 \(\Cst\)\nb-action \(\gamma\) of~\(\G\) on \(\Contvin(\widehat{G_1})\) 
 constructed in Theorem~\ref{the:Cst_coact}.

 Recall, a right action~\(\gamma\) of a von Neumann algebraic quantum 
 group~\(\Qgrp{G}{M}\) on a von Neumann algebra \(N\) is called \emph{ergodic} if 
 \(N^{\gamma}\defeq\{x\in N :\text{ \(\gamma(x)=x\otimes 1_{M}\)}\}\) 
 is equal to~\(\C\cdot 1_{N}\). 
  
 \begin{proposition}
  \label{prop:Ergodic}
  The von Neumann algebraic action \(\gamma\) of~\(\G\) 
  on~\(L(G_1)\) is ergodic.
 \end{proposition}
 \begin{proof}
 Let~\(\Bialg{M}\) be the von Neumann algebraic version of~\(\G\). 
 By construction, \(\gamma\) is obtained from the comultiplication \(\Comult[A]\) 
 of the \(\Cst\)\nb-algebraic version of \(\G\). Since, \(\Comult[A]\) extends uniquely 
 to the comultiplication~\(\Comult[M]\) on~\(M\) by~\eqref{eq:bicros_qnt_grp}. Then, 
 in a similar way, \(\gamma\) also extends to a von Neumann algebraic action, denoted again by
 \(\gamma\), of \(\G\) on~\(L(G_1)\). Now ergodicity of~\(\Comult[M]\) 
 (see~\cite{Kustermans-Vaes:LCQG}*{Result 5.13} or 
 \cite{Meyer-Roy-Woronowicz:Homomorphisms}*{Theorem 2.1}) 
 implies the same for~\(\gamma\).
 \end{proof}
 \subsection{Faithfulness}
 Motivated by~\cite{Goswami-Joardar:Rigidity_CQG_act}*{Definition 4.5} 
 we  propose the following definition, as a possible generalisation 
 of faithful actions on locally compact quantum groups on \(\Cst\)\nb-algebras.
 \begin{definition}
  \label{def:coact_faithful}
  A \(\Cst\)\nb-action \(\gamma\colon C\to C\otimes A\) of~\(\Qgrp{G}{A}\) 
  on a~\(\Cst\)\nb-algebra \(C\) is called \emph{faithful} if the 
  \Star{}algebra generated by~\(\{(\omega\otimes\Id_{A})\gamma(c) : 
  \text{\(\omega\in C'\), \(c\in C\)}\}\) is strictly dense in~\(\Mult(A)\).
 \end{definition} 
 \begin{example}
  \label{ex:comult_faithful}
  In particular, the comultiplication map~\(\Comult[A]\) is a \(\Cst\)\nb-action of~\(\G\) 
  on~\(A\). Given any~\(\omega\in A'\) and~\(a\in A\) define 
  \(\omega\cdot a\in A'\) by~\(\omega\cdot a(b)\defeq\omega(ab)\) 
  for~\(b\in A\). Now the space~\(\{\omega\cdot a :\text{\(\omega\in A'\),  
  \(a\in A\)}\}\) is weak~\Star{}dense in~\(A'\). The cancellation 
  property~\eqref{eq:cancellation} of \(\Comult[A]\) shows that 
  \(\{(\omega\otimes\Id_{A})\Comult[A](A) : \text{ \(\omega\in A\) 
  and~\(a\in A\)}\}\) is norm dense in~\(A\); hence \(\Comult[A]\) 
  is a faithful \(\Cst\)\nb-action.
 \end{example} 
 \begin{theorem}
 \label{the:faithful_Cst_coact}
 The \(\Cst\)\nb-action~\(\gamma\) of~\(\G\) on~\(\Contvin(\widehat{G_{1}})\)
  is faithful if and only if~\(\beta\) is non\nb-trivial. 
\end{theorem}
\begin{proof}
If possible, assume that \(\beta\) is trivial. Then Proposition~\ref{prop:def_gamma} 
gives 
\[
\gamma(\lambda(f))=\Comult[\Cred(G_{1})](\lambda (f))\otimes 1 
\quad\text{for all \(f\in\CompSupp(G_{1})\).}
\]
 Here \(\Comult[\Cred(G_{1})]\) denotes the comultiplication 
on~\(\Cred(G_{1})\). By Example~\ref{ex:comult_faithful} we observe that the set
\(\{(\omega\otimes\Id_{A})\gamma(\lambda(f)) :\text{ \(\omega\in \Cred(G_{1})'\), 
\(f\in\CompSupp(G_{1})\)}\}\) is strictly dense in~\(\Mult(\Cred(G_{1})\); hence 
\(\gamma\) is not faithful. Therefore, by contraposition we obtain that the faithfulness of 
\(\gamma\) implies the nontriviality of \(\beta\).

Conversely, assume that~\(\beta\) is non\nb-trivial. By virtue of 
\cite{Baaj-Skandalis-Vaes:Non-semi-regular}*{Proposition 3.6} it is 
enough to show the set \(\{(\omega\otimes\Id_{A})\gamma(\lambda(f)) :
\text{ \(\omega\in \Cred(G_{1})'\)}\}\) is norm dense in 
\(\alpha(\Contvin(G_{2}))\cdot (\Cred(G_{1})\otimes 1)\). 

Let~\(K\) and~\(K'\) be compact subsets of~\(G_{1}\) such that 
\(\mu_{1}(K)\neq 0\) and~\(\mu_{1}(K')\neq 0\). Let~\(\chi\) and 
\(\chi'\) be characteristic functions of~\(K\) and~\(K'\) respectively. 
For a given~\(f\in\CompSupp(G_{1})\) define
\[
  A_{\chi,\chi'}=(\omega_{\chi,\chi'}\otimes\Id_{A})\gamma(\lambda(f)).
\]
Here~\(\omega_{\chi,\chi'}\) denotes the contraction with respect to 
\(\chi_{K}\) and~\(\chi_{K'}\). 

For all~\(\xi\in L^{2}(G_{1}\times G_{2})\) we get,
\begin{align*}
     A_{\chi,\chi'}\xi(h,t) 
&= \int\conj{\chi(g)}\gamma(\lambda(f))(\chi'(g)\otimes\xi(h,t)\diff\mu_{1}(g) \\
&= \iint \conj{\chi(g)}f(z)\chi'(g\beta_{\alpha_{h}(t)}(z))\xi(hz,t)\diff\mu_{1}(g)\diff\mu_{1}(z)\\
&= \int \mu_{1}(K\cap\beta_{\alpha_{h}(t)}(z)K') f(z)\xi(hz,t)\diff\mu_{1}(z)
\end{align*}
Let~\(\{K_{n}\}_{n\in\mathbb{N}}\) be an increasing sequence of compact subsets of~\(G_{1}\) and growing 
up to the whole group~\(G_{1}\). Let~\(\chi_{n}\) denotes the characteristic function of 
\(K_{n}\) for all~\(n\in\mathbb{N}\). By dominated convergence theorem, 
\[
    \lim_{n\to\infty}A_{\chi_{n},\chi'}\xi(h,t) 
  = \int\mu_{1}(\beta_{\alpha_{h}(t)}(z)K') f(z)\xi(hz,t)\diff\mu_{1}(z).
\]
Invariance of Haar measure~\(\mu_{1}\) gives~\(\mu_{1}(\beta_{\alpha_{h}(t)}(z)K') 
=\mu_{1}(K')\) for all~\(h,z\in G_{1}\), \(t\in G_{2}\). Therefore, 
\[
    \lim_{n\to\infty}A_{\chi_{n},\chi'}\xi(h,t) 
  = \mu_{1}(K')\int f(z)\xi(hz,t)\diff\mu_{1}(z)
  =\mu_{1}(K')(\lambda(f)\otimes 1)\xi(h,t).
\]
Thus~\(A_{\chi_{n},\chi}\) goes strictly to~\(\lambda_{g_{0}}\) for some 
\(g_{0}\in G_{1}\). Therefore, for any~\(\eta,\eta'\in L^{2}(G_{1})\), the operator 
\(B_{\eta,\eta',g_{0}}\defeq(\omega_{\eta,\eta'}\otimes\Id_{A})\gamma(\lambda(g_{0}))(\lambda_{g_{0}^{-1}}\otimes 1)\) 
is in the desired algebra. 

Next we compute, 
\begin{align*}
     B_{\eta,\eta',g_{0}}\xi(h,t) &= 
   \big((\omega_{\eta,\eta'}\otimes\Id_{A})\gamma(\lambda(g_{0}))\big)(\lambda_{g_{0}^{-1}}\otimes 1)\xi(h,t)\\
  &= (\lambda_{g_{0}^{-1}}\otimes 1)\int \conj{\eta(g)}\eta'(g\beta_{\alpha_{h}(t)}(g_{0}))\xi(hg_{0},t)\diff\mu_{1}(g)\\     
  &= \int \conj{\eta(g)}\eta'(g\beta_{\alpha_{h}(t)}(g_{0}))\diff\mu_{1}(g) \xi(h,t),
  \qquad\text{for \(\xi\in L^{2}(G_{1}\times G_{2})\).}
\end{align*}
Since, \(\beta\) is non\nb-trivial, varying~\(\eta\), \(\eta'\), \(g_{0}\) and using 
Stone-Weierstrass theorem we get the set of functions \(x\to 
\int\conj{\eta(g)}\eta'(g\beta_{x}(g_{0}))\diff\mu_{1}(g)\) for \(x\in G_{2}\) 
which is norm dense in~\(\Contvin(G_{2})\). Therefore, the set of 
operators~\(B_{\eta,\eta',g_{0}}\) is norm dense in~\(\alpha(\Contvin(G_{2}))\). 
\end{proof}
\subsection{Isometry}
 \label{sec:Isometry} 
 We recall the definition of isometric action from \cite{Goswami:Qnt_isometry} for compact quantum groups. In analogy to this, it is natural to make the following definition of isometric action in the locally compact set-up:
  \begin{definition}
   Let \(\gamma\) be a \(\Cst\)\nb-action of a locally compact quantum group \(\G\) on 
   \(\Contvin(X)\) where \(X\) is a smooth Riemannian (possibly non\nb-compact) manifold with 
   the Hodge-Laplacian \(\mathcal{L}=-d^*d\), where \(d\) denotes the de\nb-Rham differential operator. 
   We say that \(\gamma\) is for every bounded linear functional \(\omega\) on \(M(Q)\), 
   \(({\rm id} \otimes \omega) \circ \gamma\) maps \(\Contb^\infty(X)\) to itself and commutes with 
   \(\mathcal{L}\) on that subspace.
  \end{definition}
Just as in \cite{Goswami-Joardar:Rigidity_CQG_act}, 
it is easy to prove that for any isometric action \(\gamma\) and any smooth vector field 
\(\chi\) on \(X\), \(f, \phi \in \Contb^{\infty}(X)\), \((\chi \otimes {\rm id})(\gamma(f))\) 
and \(\gamma(\phi)\) will commute. 

\begin{proposition}
  \label{Prop:Isometry}
  Assume~\(\widehat{G_1}\) is a Lie group. Then the~\(\Cst\)\nb-action \(\gamma\)
  of~\(\G\) on~\(\Contvin(\widehat{G_1})\) is isometric whenever either~\(\alpha\) 
  or \(\beta\) is trivial.  
 \end{proposition}
 \begin{proof}
  As noted before, the condition of isometry of~\(\gamma\) implies the operators 
  \((\delta_{i}\otimes\Id_{A})\gamma(\lambda_{g_{1}})\) and 
  \(\gamma(\lambda_{g_{2}})\) commute for all~\(g_{1},g_{2}\in G_{1}\) 
  and derivations~\(\delta_{i}\) on~\(\widehat{G_{1}}\). 
  
  Let~\(\xi\in L^{2}(G_{1}\times G_{1}\times G_{2})\). 
  Using Proposition~\ref{prop:def_gamma}, we compute 
  \begin{align*}
   L &=(\delta_{i}\otimes\Id_{A})\gamma(\lambda_{g_{1}})\gamma(\lambda_{g_{2}})\xi(g,h,t)\\
   &= \rho_{i}(\beta_{\alpha_{h}(t)}(g_{1}))\gamma(\lambda_{g_{2}})\xi\big(g\beta_{\alpha_{h}(t)}(g_{1}),hg_{1},t\big)\\
   &= \rho_{i}(\beta_{\alpha_{h}(t)}(g_{1}))\xi\big(g\beta_{\alpha_{h}(t)}(g_{1})\beta_{\alpha_{hg_{1}}(t)}(g_{2}), hg_{1}g_{2},t\big)
  \end{align*}
 Using~\eqref{eq:beta_comp} and commutativity of~\(G_{1}\) we get, 
 \[
  L=\rho_{i}(\beta_{\alpha_{h}(t)}(g_{1}))\xi\big(g\beta_{\alpha_{h}(t)}(g_{1}g_{2}), hg_{1}g_{2},t\big).
 \] 
 A similar computation gives 
 \[
   R =\gamma(\lambda_{g_{2}})(\delta_{i}\otimes\Id_{A})\gamma(\lambda_{g_{1}})\xi(g,h,t)
      =\rho_{i}(\beta_{\alpha_{hg_{2}}(t)}(g_{1}))\xi\big(g\beta_{\alpha_{h}(t)}(g_{1}g_{2}), hg_{1}g_{2},t\big).
 \]
 Now~\(L=R\) for all~\(\xi\in L^{2}(G_{1}\times G_{1}\times G_{2})\), \(g_{1},g_{2}\in G_{1}\), and 
 \(\rho_{i}\in\widehat{G_{1}}\), implies 
\[ 
   \beta_{\alpha_{h}(t)}(g_{1})=\beta_{\alpha_{hg_{2}}(t)}(g_{1}) 
   \qquad\text{ for all~\(g_{1},g_{2}, h\in G_{1}, t\in G_{2}\).} 
\]
 This is true if either of the actions~\(\alpha\) or \(\beta\) is trivial. 
\end{proof} 
We prove the main result of this section:
\begin{theorem}
 \label{the:rig_iso_faithful}
 Assume \(\widehat{G_{1}}\) is a Lie group and the \(\Cst\)\nb-action \(\gamma\) of 
 \(\G\) on~\(\Contvin(\widehat{G_1})\) is faithful and isometric. Then~\(\G\) is classical 
 group. 
\end{theorem}
\begin{proof}
 By Theorem~\ref{the:faithful_Cst_coact}, faithfulness of \(\gamma\) implies that 
 \(\beta\) is non\nb-trivial. On the other hand, \(\gamma\) is isometric; hence 
 Proposition~\ref{Prop:Isometry} forces \(\alpha\) to be trivial. 
 From~\eqref{eq:bicros_qnt_grp} and using the fact that~\(G_{1}\) is abelian we 
 get~\(M=L^\infty(\widehat{G_1}\times G_2)\).
\end{proof}
\section{Examples}
 \label{sec:Example}
 `\(ax+b\)' is the group of affine transformations of the real line~\(\R\). 
 The natural action of~\(ax+b\) on~\(\R\) given by~\(x\mapsto ax+b\) 
 for~\(a\in\R\setminus\{0\}\) and~\(b,x\in\R\) is faithful. 
 We apply our results on two versions of quantum~\(ax+b\) group 
 discussed in~\cite{Vaes-Vainerman: Extension_of_lcqg}*{Section 5}. 
 Both of them are genuine non\nb-compact, non\nb-discrete, non\nb-Kac quantum 
 groups. We show that they act ergodically and faithfully on non\nb-compact Riemannian 
 manifolds. However none of these actions are not isometric.
 \subsection{Baaj-Skandalis' \texorpdfstring{$\textup{ax+b}$}{ax+b} group}
  \label{subsec:ax+b}      
  Assume~\(G_{1}=G_{2}=\R\setminus\{0\}\) and the group operation 
  is the usual multiplication. Let~\(G=\{(a,b)\text{ \(\mid\) \(a\in\R\setminus \{0\}\), \(b\in\R\)}\}\) with 
  \((a,b)(c,d)=(ac,d+cb)\).  
  
  Define~\(i\colon G_{1}\mapsto G\) and \(j\colon G_{2}\mapsto G\) by 
  \[
     i(g)\defeq (g,g-1), \qquad j(s)\defeq (s,0), 
  \] 
  for all~\(g\in G_{1}\) and~\(s\in G_{2}\).
 This way \((G_{1},G_{2})\) is a matched pair in the sense of Definition~\ref{def:mathced_pair}. 
 Associated actions~\(\alpha\) and~\(\beta\) are defined by 
 \[ 
   \alpha_{g}(s)=\frac{gs}{s(g-1)+1}, 
   \qquad
   \beta_{s}(g)\defeq s(g-1)+1.
 \]
 for all~\(g,s\in \R\setminus\{0\}\) such that~\((g-1)\neq -s^{-1}\).

 Associated bicrossed product~\(\G\) is the Baaj-Skandalis' quantum 
 \(ax+b\) group (see~\cite{Vaes-Vainerman:Extension_of_lcqg}*{Section 5.3}). 
 By~\cite{Vaes-Vainerman:Extension_of_lcqg}*{Proposition 5.2 \& 5.3}, 
 \(\G\) is self dual, non\nb-Kac, non\nb-compact, non\nb-discrete quantum group. 

\begin{proposition}
 \label{prop:ax+b}
 There is an ergodic, faithful and non\nb-isometric  
 \(\Cst\)\nb-action of~\(\G\) on~\(\Contvin(\R\setminus\{0\})\).
\end{proposition} 
\begin{proof}
 Clearly, \(G_{1}\) is abelian and~\(\widehat{G_1}\) is a
 Lie group with two connected components. 
 Since~\(\beta\) is non\nb-trivial and~\(\G\) is 
 a genuine quantum group, by 
 Proposition~\ref{prop:Ergodic}, Theorem~\ref{the:faithful_Cst_coact}, and 
 Theorem~\ref{the:rig_iso_faithful} the \(\Cst\)\nb-action 
 \(\gamma\) of~\(\G\) on~\(\Contvin(\R\setminus\{0\})\) in 
 Theorem~\ref{the:Cst_coact} is ergodic, faithful and not isometric.
\end{proof}
\subsection{Split--Extension}
 \label{subsec:Split_ext} 
  Assume~\(G_{1}=\{(a,b)\text{ \(\mid\) \(a>0\), \(b\in\R\)}\}\) with 
  \((a,b)(c,d)\defeq (ac,ad+\frac{b}{c})\) and 
  \(G_{2}=(\R,+)\). 
  Let~\(K\) be the multiplicative group with two elements. Define~\(G=SL_{2}(\R)/K\), and 
  \(i\colon G_{1}\mapsto G\), \(j\colon G_{2}\mapsto G\) by 
  \[
    i(a,b)\defeq \left( \begin{array}{cc} a & b\\ 0 & \frac{1}{a}\end{array}\right) \text{ mod }K, 
    \qquad
    j(x)\defeq \left( \begin{array}{cc} 1 & 0\\ x & 1\end{array}\right)\text{ mod }K .
  \]
 This way \((G_{1},G_{2})\) is a matched pair in the sense of Definition~\ref{def:mathced_pair}. 
 Associated actions~\(\alpha\) and~\(\beta\) are defined by 
 \[
  \alpha_{(a,b)}(x)\defeq \frac{x}{a(a+bx)}, 
  \qquad
  \beta_{x}(a,b)\defeq\left\{
        \begin{array}{ll}
            (a+bx,b) & \quad\text{if~\(a+bx>0\),}\\
           (-a-bx,-b) & \quad\text{if~\(a+bx<0\).} 
        \end{array}\right .
 \]
 whenever~\(ax+b\neq 0\).
 
 By \cite{Vaes-Vainerman:Extension_of_lcqg}*{Proposition 5.5}, associated bicrossed product \(\G\) 
 and its dual \(\DuG\) are non\nb-Kac, non\nb-compact, and non\nb-discrete quantum group. Also, 
 \(\DuG\) is not unimodular. Moreover, \cite{Vaes-Vainerman:Extension_of_lcqg}*{Remark 5.6} shows 
 that~\(\DuG\) is deformation of some generalised~\(ax+b\) group.  
 \begin{proposition}
  \label{prop:ex_split_ext}
 There is an ergodic, faithful and non\nb-isometric  
 \(\Cst\)\nb-action of~\(\DuG\) on~\(\Contvin(\R)\).
 \end{proposition}
 \begin{proof}
 Clearly, \(G_{2}\) is an abelian and~\(\widehat{G_2}\) is a connected 
 Lie group. Recall that 
 \(\DuG\) is obtained by exchanging \(G_1\) and~\(\alpha\) 
 with \(G_2\) and~\(\beta\). Since~\(\alpha\) is 
 non\nb-trivial, Proposition~\ref{prop:Ergodic}, Theorem~\ref{the:faithful_Cst_coact}, and 
 Theorem~\ref{the:rig_iso_faithful} give the \(\Cst\)\nb-action 
 \(\gamma\) of~\(\DuG\) on~\(\Contvin(\R)\) in 
 Theorem~\ref{the:Cst_coact} is ergodic, faithful and not isometric.
 \end{proof}

\end{document}